\documentclass[12pt,leqno]{amsart}

\usepackage{amsmath}
\usepackage{times,amsmath,amsthm,amssymb,latexsym,amscd,bbm,color}
\usepackage{mathrsfs}
\usepackage[all]{xy}
\usepackage{verbatim}
\usepackage{graphicx}

\CompileMatrices 

\newtheoremstyle{fancy}{}{}{\itshape}{}{\textsc\bgroup}{.\egroup}{ }{}
\newtheoremstyle{fanci}{}{}{\rm}{}{\textsc\bgroup}{.\egroup}{ }{}
\theoremstyle{fancy}
\newcounter{intro}

\numberwithin{equation}{section}    \swapnumbers
\newtheorem{cor}[equation]{Corollary}
\newtheorem{lem}[equation]{Lemma}
\newtheorem{prop}[equation]{Proposition}
\newtheorem{thm}[equation]{Theorem}
\newtheorem{mt}[equation]{Main Theorem}
\newtheorem{question}[equation]{Question}
\theoremstyle{fanci}
\newtheorem{dfn}[equation]{Definition}
\newtheorem{exa}[equation]{Example}

\newtheorem{rem}[equation]{Remark}

\newcommand{\cref}[1]{Corollary~\ref{#1}}

\def\R{\mathbb{R}}    
\def\Z{{\mathbb{Z}}}  
\def\C{{\mathbb{C}}}
\def\Q{{\mathbb{Q}}}

\def\ot{\otimes}
\def\Ga{\Gamma}

\def\bs{{\backslash}}

\def\ms{\mathscr}

\def\Id{\operatorname{Id}}
\def\Res{\operatorname{Res}}
\def\Ind{\operatorname{Ind}}
\def\Hom{\operatorname{Hom}}
\def\End{\operatorname{End}}
\def\ker{\operatorname{ker}}
\def\im{\operatorname{im}}
\def\Isom{\operatorname{Isom}}

\def\ol{\overline}

\providecommand{\abs}[1]{\left\vert#1\right\vert}

\providecommand{\bmat}[1]{\begin{bmatrix}#1\end{bmatrix}}

\providecommand{\Dbd}[1]{\frac{\partial}{\partial#1}}

\providecommand{\bmat}[1]{\begin{bmatrix}#1\end{bmatrix}}

\begin{document}

\begin{title}[Isogeny of Intermediate Jacobians]{Transplantation and isogeny of intermediate Jacobians of compact K\"ahler manifolds}\end{title}

\author{Carolyn Gordon}
\address[Gordon]{Department of Mathematics, Dartmouth College, Hanover, 
New Hampshire 03755, U. S. A.}
\email{csgordon@dartmouth.edu}

\author{Eran Makover}
\address[Makover]{Department of  Mathematical sciences, Central Connecticut State University, 1615 Stanley St., New Britain, Connecticut 06050, U. S. A.}
\email{makovere@mail.ccsu.edu}

\author{Bjoern Muetzel}
\address[Muetzel]
{Department of Mathematics, Dartmouth College, Hanover, 
New Hampshire 03755, U. S. A.}
\email{bjorn.mutzel@gmail.com}

\author{David Webb}
\address[Webb]{Department of Mathematics, Dartmouth College, Hanover, 
New Hampshire 03755, U. S. A.}
\email{david.l.webb@dartmouth.edu}





\begin{abstract}
We give a general method for constructing compact K\"ahler manifolds $X_1$ and $X_2$ whose intermediate Jacobians $J^k(X_1)$ and $J^k(X_2)$ are isogenous for each $k$, and we exhibit some examples.  The method is based upon the algebraic transplantation formalism arising from Sunada's technique for constructing pairs of compact Riemannian manifolds whose Laplace spectra are the same.  We also show that the method produces compact Riemannian manifolds whose Lazzeri Jacobians are isogenous.
\end{abstract}

\maketitle

\tableofcontents

\section{Introduction} 

One of the classical invariants of a compact Riemann surface $X$ of genus $g$ is its Jacobian $J(X)$, a complex torus of complex dimension $g$ given as the quotient of the vector space of complex linear functionals on the space of holomorphic $1$-forms by the lattice of periods, those linear functionals arising from integration over $1$-cycles.  More generally, let $X$ be a compact, complex manifold of complex dimension $d$.  We say $X$ is of \emph{K\"ahler type} if it admits a K\"ahler structure.  In this case, for each odd integer $k$ in the range $1\le k\le d$, one associates to $X$ a complex torus called the $k$th Griffiths \emph{intermediate Jacobian} $J_G^k(X)$.  For $k=1$, it is the classical Picard variety, while for $k=d$, it is the Albanese variety.  (If $d=1$, then the Picard and Albanese varieties are both just the Jacobian of the Riemann surface $X$.)  One can also associate to a compact K\"ahler manifold $X$ its $k$th Weil intermediate Jacobian $J_W^k(X)$; it is the same topological torus as $J_G^k(X)$, but endowed with a different complex structure.  While the existence of a K\"ahler structure is necessary for their existence, the intermediate Jacobians are independent of the choice of K\"ahler structure; they depend only on the complex structure.

Our main result is to give a technique, with many examples, for constructing pairs $X_1$ and $X_2$ of compact complex manifolds of K\"ahler type whose Griffiths intermediate Jacobians $J_G^k(X_1)$ and $J_G^k(X_2)$ are isogenous for each $k$; the Weil intermediate Jacobians $J_W^k(X_1)$ and $J_W^k(X_2)$ are also isogenous.   Two subgroups $\Ga_1$ and $\Ga_2$ of a finite group $G$ are said to be \emph{almost conjugate} if every $G$-conjugacy class intersects $\Ga_1$ and $\Ga_2$ in the same number of elements; thus there is a set bijection $\Ga_1\to\Ga_2$ that carries each element of $\Ga_1$ to an element of $\Ga_2$ conjugate to it in $G$.  

\begin{mt}\label{thm.mt}
Let $X$ be a compact complex manifold of K\"ahler type and let $G$ be a finite group of biholomorphic transformations of $X$.  Suppose that $\Ga_1$ and $\Ga_2$ are  almost conjugate subgroups of $G$ that act freely on $X$, and let $X_i=\Ga_i\bs X$ for $i=1,2$.  Then for $1\le k\le d$, the Griffiths intermediate Jacobians $J_G^k(X_1)$ and $J_G^k(X_2)$ are isogenous.  The Weil intermediate Jacobians $J_W^k(X_1)$ and $J_W^k(X_2)$ are also isogenous.
\end{mt}

If the subgroups $\Ga_1$ and $\Ga_2$ are actually conjugate in $G$, then $X_1$ and $X_2$ will be biholomorphically equivalent, and the statement of the theorem becomes trivial.  When the subgroups are only almost conjugate, the quotients $X_1$ and $X_2$ are not \emph{a priori} biholomorphically equivalent, but the possibility of an ``accidental'' biholomorphic equivalence must be ruled out.  Of course, if the subgroups $\Ga_1$ and  $\Ga_2$ are non-isomorphic, then $X_1$ and $X_2$ have different fundamental groups, so the construction is guaranteed to yield non-trivial examples in that case.

Theorem~\ref{thm.mt} is an adaptation to the setting of K\"ahler manifolds and their Jacobians of a technique of T. Sunada \cite{Sunada} originally developed for the construction of compact Riemannian manifolds with isospectral Laplacians.   Sunada's theorem (inspired by Gassmann's \cite{Gassmann} construction of a pair of nonisomorphic algebraic number fields having the same Dedekind zeta function and hence the same arithmetic) is given as follows:

\begin{thm}(T. Sunada)\label{thm.sun}
Let $(M,g)$ be a compact Riemannian manifold, and let $G$ be a finite subgroup of the isometry group $\Isom(M,g)$.  Suppose that $\Ga_1$ and $\Ga_2$ are  almost conjugate subgroups of $G$ that act freely on $M$.  Then the orbit spaces $\Ga_1\backslash M$ and $\Ga_2\backslash M$ (which are normally covered by $M$ and hence inherit natural Riemannian metrics from $M$) are isospectral, i.e., the associated Laplace-Beltrami operators have the same eigenvalue spectrum.
\end{thm}

The hypothesis of Sunada's Theorem further guarantee that the manifolds are \emph{strongly isospectral}, i.e., not only the Laplace-Beltrami operators but all natural elliptic operators on the two manifolds are isospectral.  For example, the Hodge Laplacians on $p$-forms on $\Ga_1\backslash M$ and $\Ga_2\backslash M$ are isospectral for each $p$.

In Sunada's theorem, one assumes that the action of $G$ on $X$ preserves the Riemannian metric, and one obtains isospectrality of all operators naturally associated with the metric.  In Theorem~\ref{thm.mt}, we instead assume that $G$ preserves the complex structure in order to relate the Jacobians, which are an invariant of the complex structure.   If we endow $X$ with a K\"ahler structure and assume that $G$ preserves \emph{both} both the metric and the complex structure, then in addition to the conclusions of Theorems~\ref{thm.mt} and \ref{thm.sun}, we further obtain isospectrality of all elliptic operators naturally associated with K\"ahler structures.   In particular, on a K\"ahler manifold $X$, the Hodge Laplacian on the space $\ms{A}^k(X)$ of smooth $k$-forms preserves the subspace $\ms{A}^{p,q}(X)$ of forms of type $(p,q)$ for each pair of non-negative integers $(p,q)$ with $p+q=k$.   We show:

\begin{thm}\label{thm.pq} Let $X$ be a compact K\"ahler manifold of real dimension $n$ and let $G$ be a finite group of holomorphic isometries of
$X$.  Suppose that $\Ga_1$ and $\Ga_2$ are  almost conjugate subgroups of $G$ that act freely on $X$ and let $X_i=\Ga_i\bs X$, $i=1,2$.  Then for each pair of non-negative integers $(p,q)$ with $p+q\leq n$, the Hodge Laplacians on $\ms{A}^{p,q}(\Ga_1\bs X)$ and $\ms{A}^{p,q}(\Ga_2\bs X)$ are isospectral.
\end{thm}

We note that our results are not vacuous: a theorem of Serre \cite{Serre} shows that \emph{any} finite group $G$ is the fundamental group of a compact K\"ahler manifold (see also \cite{AmorosBurgerCorletteKotschickToledo}).  Thus one can easily use Sunada's construction to exhibit pairs of isospectral but nonisometric compact K\"ahler manifolds satisfying the hypotheses of Theorems ~\ref{thm.mt} and \ref{thm.pq}.  We list a few such examples in Section \ref{sec.examples}.

In the final section, we prove an analogous result concerning isogeny  of Lazzeri Jacobians.  F. Lazzeri (unpublished) associated a principally polarized abelian variety $J_L(M)$ to \emph{every} compact oriented Riemannian manifold $M$ whose dimension is twice an odd integer.   In particular, the definition of the Lazzeri Jacobian does not require that $M$ be of K\"ahler type, or even that it admit a complex structure.  E. Rubei \cite{Rubei} considered the Torelli and Schottky problems for the Lazzeri Jacobian.  In case $M$ is K\"ahler, she also compared it to the Griffiths and Weil intermediate Jacobians.   We prove:

\begin{thm}\label{thm.laz} Let $(M,g)$ be a compact orientable Riemannian manifold of dimension $2m$, where $m$ is an odd positive integer, and let $G$ be a finite group of orientation-preserving isometries of $(M,g)$.  Suppose that $\Ga_1$ and $\Ga_2$ are  almost conjugate subgroups of $G$ that act freely on $M$ (so the orbit manifolds $M_1=\Ga_1\backslash M$ and $M_2=\Ga_2\backslash M$ are isospectral by Sunada's Theorem).  Then the Lazzeri Jacobians $J_L(M_1)$ and $J_L(M_2)$ are isogenous.
\end{thm}

Note that the hypotheses of this theorem are identical to those of Sunada's Theorem except for the additional assumption on the dimension of $M$ and the requirement that the elements of $G$ be orientation-preserving.

\subsection{Historical context}  Sunada's Theorem has been enormously influential in the construction of isospectral manifolds.  Isospectral manifolds constructed by this technique are often referred to as \emph{Sunada isospectral} manifolds.  Among the many examples of Sunada isospectral manifolds are Riemann surfaces of every genus $g\geq 4$; see the work of P. Buser \cite{Busertrans}, R. Brooks and R. Tse \cite{BrooksTse}, and others.  H. McKean \cite{McKean} showed that every collection of mutually isospectral Riemann surfaces is finite.  Buser \cite{Buser} gave the explicit upper bound $e^{720g^2}$ on the size of any collection of isospectral Riemann surfaces of genus $g$, and recently, H. Parlier \cite{Parlier} obtained the improved upper bound $g^{154g}$.  However, huge sets of mutually isospectral Riemann surfaces do indeed exist: R. Brooks, R. Gornet, and W. Gustafson \cite{BrooksGornetGustafson} used Sunada's technique to construct such sets of order $g^{c\log(g)}$.  More generally, given any compact locally symmetric space $X$ of noncompact type, D. B. McReynolds \cite{McReynolds} constructed arbitrarily large mutually Sunada isospectral compact locally symmetric spaces with universal covering $X$.   

Sunada's technique was the first general technique for constructing isospectral manifolds and, due to its simplicity and power, it remains the most widely used method.  For completeness, we note that an arithmetic method of constructing isospectral but nonisometric Riemann surfaces introduced by M.-F. Vign\'eras \cite{Vigneras} and recently studied systematically by B. Linowitz and J. Voight \cite{LinowitzVoight} provides many examples of isospectral Riemann surfaces that do not arise from Sunada's construction.  Outside of the setting of Riemann surfaces, other methods, including a general technique involving torus actions, have produced many interesting examples.

Constructions of isospectral but nonisometric manifolds enable one to identify geometric properties of Riemannian manifolds that are not spectrally determined.  Around 1990, W. Abikoff asked a natural spectral question about Riemann surfaces:  

\begin{question}
What is the relationship (if any) between the Jacobians of two isospectral Riemann surfaces?  
\end{question}

One partial answer to this question is given by a theorem of D. Prasad and C.S. Rajan \cite{PrasadRajan}.  (See also \cite{GordonMakoverWebb}.) 

\begin{thm}\label{thm.PR} (D. Prasad and C.S. Rajan) If two compact Riemann surfaces $M_1$ and $M_2$ are Sunada isospectral, then their Jacobians are isogenous.  
\end{thm}

The hypotheses of Sunada's Theorem~\ref{thm.sun} and our main theorem~\ref{thm.mt} are essentially the same in the case of hyperbolic Riemann surfaces, since biholomorphic maps are precisely the orientation-preserving isometries of the metric.  In the case of K\"ahler manifolds of higher dimension, it makes no sense to ask whether isospectral manifolds  have isogenous intermediate Jacobians, since the intermediate Jacobians depend only on the complex structure while the Laplace spectra depend only on the metric.  Instead, Theorems~\ref{thm.mt} and \ref{thm.laz} may be viewed as two natural generalizations of Theorem~\ref{thm.PR} on Riemann surfaces, the former focusing on the complex structure and the latter on the Riemannian metric.

\subsection{Transplantation}\label{trans}  Two subgroups $\Ga_1$ and $\Ga_2$ of a finite group $G$ are said to be \emph{representation equivalent in $G$} if the two linear representations $\C[\Ga_1\bs G]$ and $\C[\Ga_2\bs G]$ obtained from the right action of $G$ on the right coset spaces $\Ga_1\bs G$ and $\Ga_2\bs G$ are equivalent as linear (right) representations of $G$.  For a subgroup $\Ga\subseteq G$, the representation $\C[\Ga\bs G]$ is just the right representation $\Ind_{\Ga}^G(1_{\Ga})$ of $G$ obtained by inducing up to $G$ the trivial one-dimensional complex representation $1_{\Ga}$ of $\Ga$; thus the representation equivalence condition can be restated as $\Ind_{\Ga_1}^G(1_{\Ga_1})\cong\Ind_{\Ga_2}^G(1_{\Ga_2})$ as right representations of $G$.  As already observed by Gassmann \cite{Gassmann}, the character theory of representations of finite groups implies that $\Ga_1$ and $\Ga_2$ are representation equivalent in $G$ if and only if they satisfy the almost conjugacy condition of Sunada's Theorem.  There are numerous proofs of Sunada's Theorem, some directly using the almost conjugacy condition, others using the representation equivalence of $\Ga_1$ and $\Ga_2$.

P. Buser \cite{Busertrans} and S. Zelditch \cite{Zelditch} noted that the representation equivalence of $\Ga_1$ and $\Ga_2$ in Sunada's theorem furnishes a ``transplantation,'' an explicit map that ``transplants'' a $\lambda$-eigenfunction on $\Ga_1\bs M$ to a $\lambda$-eigenfunction on $\Ga_1\bs M$.  Transplantation was clarified and studied more systematically by P. B\'erard \cite{Berard1}, \cite{Berard2}, and simplified further in \cite{BrooksGornetPerry}.  The essential simplicity of the underlying algebra was clarified in \cite{GordonMakoverWebb}, where the algebraic transplantation of the modules of invariants and coinvariants of a $G$-module were used to give a geometric proof of Prasad and Rajan's Theorem~\ref{thm.PR}, and to provide another partial answer to the question posed by Abikoff.  The current paper is motivated by and follows the methods of \cite{GordonMakoverWebb}.  

\begin{rem} We wish to underscore the point that Theorem \ref{thm.mt} has nothing to do with isospectrality \emph{per se}, but is an application of the method of transplantation.  Although transplantation arose originally in the construction of isospectral manifolds, we anticipate that, because of its algebraic simplicity that we have tried to convey, transplantation should be more broadly applicable.  In particular, transplantation of coinvariants first arose in the construction of isogenous Jacobians, not in the setting of spectral geometry.
\end{rem}
\subsection{Acknowledgments}  The first and fourth authors thank W. Abikoff for interesting them in Jacobians, and Jes\'us Ruiz and the Universidad Complutense of Madrid for their hospitality during a sabbatical term, when this work was conceived.  The second author thanks Dartmouth College for its hospitality.

\section{Notation} For a smooth manifold $X$, we denote by $\ms{A}^n(X)$ the space of real-valued smooth $n$-forms on $X$.  For a real vector space $V$, we denote its complexification $\C\otimes_{\R}V$ by $V_{\C}$; in particular, the space of complex-valued $n$-forms on $X$ is denoted by $\ms{A}^n(X)_{\C}$ or by $\ms{A}^n_{\C}(X)$.  If $V$ is a real vector space equipped with a complex structure $I$ (i.e., an endomorphism $I\in\End_{\R}(V)$ satisfying $I^2=-\Id$), then its complexification $V_{\C}$ admits a decomposition $V_{\C}=V^{1,0}\oplus V^{0,1}$ into complex subspaces $V^{1,0}$ and $V^{0,1}$, where $V^{1,0}$ (respectively, $V^{0,1}$) is the $i$-eigenspace (respectively, the $(-i)$-eigenspace) of the complex linear extension $I_{\C}$ of $I$.  When $V$ is the cotangent space $T^*_xX$ of a complex manifold $X$, this decomposition of its complexification $(T^*_xX)_{\C}=(T^*_xX)^{1,0}\oplus (T^*_xX)^{0,1}$ gives rise to the usual bigrading by type on the complex-valued smooth forms: $\ms{A}_{\C}^n(X)=\coprod_{p+q=n}\ms{A}^{p,q}(X)$; in local holomorphic coordinates $z_1,\dotsc,z_d$, a form $\eta\in\ms{A}^{p,q}(X)$ has the coordinate expression 
$$\eta=\sum_{1\le i_1<\dotsc,i_p<d,1\le j_1<\dotsc<j_q\le d}f_{i_1,\dotsc,i_p;j_1,\dotsc,j_q}dz_{i_1}\wedge\dots\wedge dz_{i_p}\wedge d\bar{z}_{j_1}\wedge\dotsc\wedge d\bar{z}_{j_q}$$
where $d$ is the complex dimension of $X$.

Note that we denote the integral homology and cohomology of a space $X$ by $H_n(X)$ and $H^n(X)$, but by $H^n(X,\Z)$ we mean the quotient of $H^n(X)$ by its torsion subgroup.  For a space $X$ with finitely generated integral homology, this means by the universal coefficient theorem that, since the kernel of the natural surjection $H^n(X,\Z)\to\Hom(H_n(X),\Z)$ is torsion, $H^n(X,\Z)$ is naturally identified with $\Hom(H_n(X),\Z)$.  

\section{K\"ahler manifolds}  In this section we recall some basic definitions and establish notation.  Standard sources for this material are \cite{GriffithsHarris}, \cite{Huybrechts}, \cite{Voisin}, and \cite{Wells}. 

Let $X$ be a complex manifold, with complex structure denoted by $I$, so $I^2=-\Id$.  Recall that a Riemannian metric $g$ on $X$ is a \emph{Hermitian structure} if for any point $x\in X$, the endomorphism $I_x$ of the tangent space $T_xX$ is a linear isometry relative to the Riemannian inner product $g_x$ on $T_xX$: thus for any tangent vectors $v,w\in T_xX$,
\begin{equation}\label{IOrthog}
g(I(v),I(w))=g(v,w).
\end{equation}
Then the \emph{fundamental $2$-form} $\omega$ on $X$ is defined by $\omega(v,w)=g(I(v),w)=-g(v,I(w))$.  Then the Riemannian inner product and the fundamental $2$-form are essentially the real and imaginary parts of a Hermitian metric on $X$; more precisely, $h(v,w)=g(v,w)-i\omega(v,w)$ defines a positive-definite Hermitian inner product $h$ on $T_xX$.  Any two of the data $g$, $I$, and $\omega$ determine the third.  Recall the definition
\begin{dfn}
A complex Hermitian manifold $X$ is \emph{K\"ahler} if its fundamental $2$-form $\omega$ is closed: $d\omega=0$.
\end{dfn}
A convenient equivalent condition is that for every point $x\in X$, there is a holomorphic coordinate system $z_1,\dotsc,z_d$ centered at $x$ such that the Hermitian metric $h$ osculates to second order with the standard Hermitian metric on $\C^d$: i.e., the matrix $H=[h_{ij}]_{ij}=[h(\Dbd{z_i},\Dbd{\ol{z}_j})]_{ij}$ representing the Hermitian form $h$ in these coordinates has the form $H=I_d+O(\sum_i\abs{z_i}^2)$ (see e.g. \cite{Voisin}, Proposition 3.14).  A consequence is that the bigrading $\ms{A}^k(X)=\coprod_{p+q=k}\ms{A}^{p,q}(X)$ on differential forms descends to a bigrading on cohomology, the Hodge decomposition:
\begin{prop}\label{HodgeDecomp}
Let $X$ be a compact K\"ahler manifold.  For each $n$, there is a direct sum decomposition as $\C$-vector spaces
$$H^n(X,\C)=\coprod_{p+q=n}H^{p,q}(X),$$
whose summands $H^{p,q}(X)$ are independent of the K\"ahler structure.  The complex structure on $H^n(X,\R)$ arising from the complex manifold structure on $X$ defines an $\R$-linear conjugation involution on $H^n(X,\R)$ whose $\C$-linear extension to $H^n(X,\C)$ interchanges $H^{p,q}(X)$ and $H^{q,p}(X)$: 
$$\ol{H^{p,q}(X)}=H^{q,p}(X).$$
\end{prop}
\begin{proof}
See, e.g., \cite{Huybrechts}, Corollary 3.2.12 or \cite{Voisin}, Proposition 6.11.
\end{proof}

\begin{dfn}
Let $X$ be a compact K\"ahler manifold.  The \emph{Hodge filtration} on the $n$th complex cohomology of $X$ is the filtration whose $p$th level is defined as 
$$F^p H^n(X,\C)=\coprod_{r\ge p}H^{r,n-r}(X)=H^{p,n-p}(X)\oplus H^{p+1,n-p-1}(X)\oplus\dotsc\oplus H^{n,0}(X).$$
\end{dfn}
Because of the symmetry $\ol{H^{p,q}(X)}=H^{q,p}(X)$, the Hodge decomposition \eqref{HodgeDecomp} takes the form 
$$H^n(X,\C)=\underbrace{H^{n,0}(X)\oplus\dotsc\oplus H^{p,n-p}(X)}_{F^p H^n(X,\C)}\oplus\underbrace{H^{p-1,n-p+1}(X)\oplus\dotsc\oplus H^{0,n}}_{\ol{F^{n-p+1}H^n(X,\C)}},$$
i.e.,
\begin{equation}\label{HodgeFiltDecomp}
H^n(X,\C)=F^p H^n(X,\C)\oplus\ol{F^{n-p+1}H^n(X,\C)}.
\end{equation}
The Hodge decomposition and the Hodge filtration determine one another; indeed, the Hodge decomposition is the associated graded module of the cohomology filtered by the Hodge filtration.

We record for subsequent use the following:
\begin{prop}\label{FiltrationCompat}
For each $n$ and $p$, let $F^p\ms{A}_{\C}^n(X)$ denote the space of complex-valued differential $n$-forms on $X$ of type $(r,n-r)$ where $r\ge p$, i.e., $F^p\ms{A}_{\C}^n(X)=\coprod_{r\ge p}\ms{A}_{\C}^{r,n-r}(X)$.  Then
$$F^p H^n(X,\C)=\frac{\ker(d:F^p\ms{A}_{\C}^n(X)\to F^p\ms{A}_{\C}^{n+1}(X))}{\im(d:F^p\ms{A}_{\C}^{n-1}(X)\to F^p\ms{A}_{\C}^n(X))}.$$
\end{prop}
\begin{proof}
See \cite{Voisin}, Proposition 7.5.
\end{proof}

\section{Intermediate Jacobians}

Let $X$ be a compact K\"ahler manifold of complex dimension $d$.  Let $k$ be an integer in the range $1\le k\le d$.  For each $j$, let $b_j(X)=\dim_{\R}(H^j(X,\R))$.  Then the Hodge filtration of the odd-dimensional cohomology groups $H^{2k-1}(X,\C)=H^{2k-1}(X,\R)_{\C}$ yields (setting $n=2k-1$ and $p=k$ in \eqref{HodgeFiltDecomp}) a direct sum decomposition $$H^{2k-1}(X,\C)=F^k H^{2k-1}(X)\oplus\ol{F^k H^{2k-1}(X)}.$$
Since $H^{2k-1}(X,\R)\cap F^k H^{2k-1}(X,\C)=0$ and the real dimensions of  $H^{2k-1}(X,\R)$ and $F^k H^{2k-1}(X,\C)$ are both $b_{2k-1}(X)$, it follows that the composite mapping
$$H^{2k-1}(X,\R)\hookrightarrow H^{2k-1}(X,\C)\to H^{2k-1}(X,\C)/F^k H^{2k-1}(X,\C)$$ 
is an isomorphism of real vector spaces.  Let $H^{2k-1}(X,\Z)$ denote the integer cohomology of $X$ modulo torsion.  Then the image under this composite mapping of the full lattice $H^{2k-1}(X,\Z)\subseteq H^{2k-1}(X,\R)$ is a full lattice in $H^{2k-1}(X,\C)/F^k H^{2k-1}(X,\C)$.  Since the latter is a complex vector space, its quotient modulo the lattice is a complex torus whose complex dimension is $\frac{1}{2}b_{2k-1}(X)$.
\begin{dfn}\label{def.jac} (\cite{GriffithsI}; see also \cite{Voisin} or \cite{CarlsonMueller-StachPeters})
The $k$th \emph{(Griffiths) intermediate Jacobian} of $X$ is the complex torus
$$J^k(X)=\frac{H^{2k-1}(X,\C)}{F^k H^{2k-1}(X,\C)\oplus H^{2k-1}(X,\Z)}.$$
\end{dfn}

\begin{rem}\label{rem.jac} A complex manifold $X$ is said to be of \emph{K\"ahler type} if it admits a K\"ahler structure.  For $X$ of K\"ahler type, the Hodge decomposition of $H^n(X,\C)$ and thus the intermediate Jacobians are independent of the choice of K\"ahler structure; they depend only on the complex structure.
\end{rem}

\begin{rem}
Perhaps the most familiar definition of the classical Jacobian of a Riemann surface $X$ is as the quotient $\Omega^1(X)^*/H_1(X)$ of the space of linear functionals on the space $\Omega^1(X)$ of holomorphic $1$-forms on $X$ modulo the periods (integrals over $1$-cycles).  For a compact K\"ahler manifold $X$ of complex dimension $d$, the top intermediate Jacobian $J^d(X)$ recovers this classical definition, as follows.

The perfect cup product pairing $H^{2k-1}(X,\C)\times H^{2d-2k+1}(X,\C)\to H^{2d}(X,\C)\cong\C$ induces the Poincar\'e duality isomorphism $H^{2k-1}(X,\C)\to(H^{2d-2k+1}(X,\C))^*$, which identifies the complex dual $(F^{d-k+1}H^{2d-2k+1}(X,\C))^*$ of the subspace $F^{d-k+1}H^{2d-2k+1}(X,\C)$ with the quotient $H^{2k-1}(X,\C)/F^{d-k+1}H^{2d-2k+1}(X,\C)^{\perp}$ by the orthogonal complement $F^{d-k+1}H^{2d-2k+1}(X,\C)^{\perp}\subseteq H^{2k-1}(X,\C)$.  But the orthogonal complement is given by $F^{d-k+1}H^{2d-2k+1}(X,\C)^{\perp}=F^k H^{2k-1}(X,\C)$. (Indeed if $\alpha$ is a $(2d-2k+1)$-form of bigrading $(r,2d-2k+1-r)$ with $r\ge d-k+1$ and $\beta$ is a $(2k-1)$-form of bigrading $(s,2k-1-s)$ with $s\ge k$, then $\alpha\wedge\beta$ has bigrading $(r+s,2d-r-s)$.  The fact that $r+s$ exceeds the complex dimension $d$ of $X$ implies  that $\alpha\wedge\beta=0$, since the local representation of $\alpha\wedge\beta$ in holomorphic coordinates $z_1,\dotsc,z_d$ would involve the wedge product of more than $d$ of the forms $dz_1,\dotsc,dz_d$.)  Thus Poincar\'e duality identifies $(F^{d-k+1}H^{2d-2k+1}(X,\C))^*$ with $H^{2k-1}(X,\C)/F^k H^{2k-1}(X,\C)$ and $H^{2k-1}(X,\Z)$ with $H_{2d-2k+1}(X)$, yielding the alternative description 
$$J^k(X)=(F^{d-k+1}H^{2d-2k+1}(X,\C))^*/H_{2d-2k+1}(X)$$
of the intermediate Jacobian.  In particular, in the case $k=d$, one obtains $J^d(X)=(F^1 H^1(X,\C))^*/H_1(X)=H^{1,0}(X)^*/H_1(X)$, which by the Dolbeault isomorphism is $H^0(X,\Omega^1(X))^*/H_1(X)=\Omega^1(X)^*/H_1(X)$.  
\end{rem}


\begin{dfn} \label{isogeny}
Let $T_1$ and $T_2$ be complex tori.  An {\emph{isogeny}} is a surjective
homomorphism $T_1 \to T_2$ with finite kernel.  If such an isogeny exists, we
say that $T_1$ and $T_2$ are {\emph{isogenous}}.  Isogeny is an equivalence
relation (see e.g. \cite{BirkenhakeLange}, Corollary 1.2.7).
\end{dfn}

\section{Algebraic Transplantation}
In this section, we recall from \cite{GordonMakoverWebb} some simple formulas for the algebraic transplantation of invariants and coinvariants in the group-theoretic context of Sunada's theorem, that of two representation equivalent subgroups of a finite group $G$ as in Subsection~\ref{trans}.  For the convenience of the reader, we have included some material from \cite{GordonMakoverWebb}, since the explanation there serves to motivate the existence of transplantations in the Gassmann-Sunada setting; however, the impatient reader may skip from the first paragraph of subsection \ref{transinvsec} directly to Theorem \ref{transinv} and from the first paragraph of subsection \ref{transcoinvsec} to Theorem \ref{transcoinv} without any loss, since it is straightforward to verify directly that the explicit formulas for transplantation of invariants and coinvariants satisfy the desired conditions.

\subsection{Notation and preliminaries}\label{notation}
Let $G$ be a group, and $R$ a commutative ring.  We denote by $RG$ the group algebra of $G$ over $R$; as an $R$-module it is free on $G$, with the algebra multiplication defined by the group multiplication on the basis $G$.  

A left (respectively, right) module $V$ over $RG$ is just an $R$-module $V$ on
which $G$ acts on the left (respectively, on the right) $R$-linearly.  We will call an $RG$-module simply a $G$-module if the coefficient ring $R$ is understood.  By a $G$-{\emph{map}} we mean an $RG$-module map, i.e., an $R$-linear $G$-equivariant map.

Given a right $RG$-module $W$, the module of $\emph G$-{\emph {invariants}} is the submodule defined by $W^G=\{w\in W: \forall g \in G, w.g=w\}$; it is the largest submodule of $W$ on which $G$ acts trivially.  Dually, given a left $RG$-module $V$, the module of $\emph G$-{\emph {coinvariants}} is the quotient defined by $V_G=V/S$, where $S$ is the $R$-submodule generated by $\{g.v-v:g \in G, v \in V\}$; it is the largest quotient module of $V$ on which $G$ acts trivially. 

\begin{rem}\label{mean}  If $W$ is a right $G$-module and $\Ga \subseteq G$ is a subgroup, then for $w \in W^{\Ga}$ and $x \in G$, the notation $w.\Ga x$ is meaningful --- indeed, the value of $w.x$ depends only upon the right coset $\Ga x$ of $x$, since for any $\gamma \in \Ga$, $w.\gamma x=w.x$, as $w$ is $\Ga$-invariant.  Similarly, for a left $G$-module $V$ and $v \in V$, the residue class $\overline{\Ga x.v}$ in $V_{\Ga}$ makes sense --- the value of $\overline{x.v}$ in $V_{\Ga}$ depends only upon the right coset $\Ga x$ of $x$, since $\overline{\gamma x.v}-\overline{x.v}=\overline{\gamma .(xv)-xv}$, which is zero in $V_{\Ga}$.
\end{rem}

For a left (respectively, right) $G$-set $X$, $R[X]$ denotes the left (respectively, right) $RG$-module that is free on $X$ as $R$-module; the $G$-action on $R[X]$ is that determined by the action of $G$ on $X$.  

\subsection{Transplantation of invariants}\label{transinvsec}

We first record a simple explicit formula for transplantation of invariants.  Let $G$ be a group, let $\Ga_1$, $\Ga_2$ be subgroups of $G$, and let $W$ be any right $RG$-module.  Let $\tau:R[\Ga_1 \bs G] \to R[\Ga_2 \bs G]$ be a map of right $G$-modules.  We will show that $\tau$ induces a contravariant functorial $R$-linear map $\tau^{\sharp }:W^{\Ga_2} \to W^{\Ga_1}$ on invariants.  We first recall two
simple isomorphisms.  

Let $\Ga$ be a subgroup of $G$, and let $W$ be a right $R\Ga$-module.  Now $RG$ is naturally a left $R\Ga$-module, and since the Hom functor is contravariant in its first argument, this left action makes the $R$-module of right $R\Ga$-module maps $\Hom_{R\Ga}(RG,W)$ (often called the {\emph{coinduced module}}) into a {\emph{right}} $G$-module via $(f.g)(x)=f(g.x)$ for $g \in G$, $x \in RG$, $f \in \Hom_{R\Ga}(RG,W)$.

\begin{prop}\label{Shapiroinv}
There is a natural $R$-linear isomorphism 
$$W^{\Ga} \to \Hom_{R\Ga}(RG,W)^G$$
of the $\Ga$-invariants of $W$ with the $G$-invariants of the coinduced module.  Explicitly, given $f \in \Hom_{R\Ga}(RG,W)$, associate the element $f(1) \in W^{\Ga}$; given $w \in W^{\Ga}$, associate the function $c_w\in\Hom_{R\Ga}(RG,W)$ which takes the constant value $w$ on every element of $G$.
\end{prop}

\begin{proof}  This is a special case of Shapiro's Lemma; see \cite{Brown}.
Alternatively, the explicit formulas are easily seen to furnish inverse 
isomorphisms.
\end{proof}

Now let $\Ga \subseteq G$, and let $W$ be a right $RG$-module; then $W$ becomes a right $R\Ga$-module by restricting scalars to $R\Ga$.  Denote $W$ viewed as a right $\Ga$-module in this way by $\Res^G_{\Ga}W$.  We can then form the coinduced module $\Hom_{R\Ga}(RG,\Res^G_{\Ga}W)$ as above; it has a right $G$-module action given by $(f.g)(x)=f(g.x)$ for $g \in G$, $x \in RG$.  

We could also form the $R$-module of $R$-linear maps $\Hom_R(R[\Ga \bs G],W)$.  Since both $R[\Ga \bs G]$ and $W$ are equipped with right $G$-actions, we can endow $\Hom_R(R[\Ga \bs G],W)$ with a right $G$-action (the \emph{diagonal action}) that takes both the $G$-actions on the domain and on the codomain into account by defining $(f.g)(\Ga x)=f(\Ga xg^{-1})g$.

\begin{prop}\label{twoactinv}
There is a natural right $G$-module isomorphism 
$$\Hom_{R\Ga}(RG,\Res^G_{\Ga}W) \to \Hom_R(R[\Ga \bs G],W).$$  It is given
explicitly as follows.  If $f \in \Hom_R(R[\Ga \bs G],W)$, associate
the element $\hat f$ of $\Hom_{R\Ga}(RG,\Res^G_{\Ga}W)$ given by
$\hat{f}(x)=f(\Ga x^{-1})x$, for $x \in G$.  If $h$ is an element of
$\Hom_{R\Ga}(RG,\Res^G_{\Ga}W)$, define 
$\check h \in \Hom_R(R[\Ga \bs G],W)$ by $\check{h}(\Ga x)=h(x^{-1})x$.
\end{prop}

\begin{proof}
This result too is standard, and is easily verified from the explicit formulas.
\end{proof}

\begin{cor}\label{twoactinvG}
There is a natural isomorphism $$\Hom_{R\Ga}(RG,\Res^G_{\Ga}W)^G \to
\Hom_R(R[\Ga \bs G],W)^G$$.
\end{cor}

\begin{proof}
Since the isomorphism of Proposition \ref{twoactinv} is a $G$-isomorphism, we
need only pass to $G$-invariants on both sides.
\end{proof}

\begin{prop}\label{compinv}
Let $\Ga$ be a subgroup of $G$, $W$ a right $G$-module.  Then there is a natural isomorphism 
$$\varphi_{\Ga}:W^\Ga \to \Hom_R(R[\Ga \bs G],W)^G$$ given as follows.  To a 
$\Ga$-invariant $w \in W^{\Ga}$, associate the mapping 
$\varphi_{\Ga}(w)=f_w \in \Hom_R(R[\Ga \bs G],W)^G$ given by
$f_w(g)=w.g=w.\Ga g$ (see Remark \ref{mean}), for $g \in G$.  
To a $G$-invariant map 
$f \in \Hom_R(R[\Ga \bs G],W)^G$, associate the element $f(\Ga)$ obtained by evaluating $f$ on the trivial right coset $\Ga$ itself.
\end{prop}

\begin{proof}
One merely composes the isomorphisms of Proposition \ref{Shapiroinv} and
Corollary \ref{twoactinvG}.  Or, one can verify this directly from the explicit
formula for $\varphi_{\Ga}$ and its inverse.
\end{proof}

\begin{thm}\label{transinv}
(Transplantation of invariants).  Let $\Ga_1$ and $\Ga_2$ be subgroups of a group
$G$, and let $W$ be a right $RG$-module.  Suppose that there is a map of right $G$-modules $\tau:R[\Ga_1 \bs G] \to R[\Ga_2 \bs G]$.  Then there is a (contravariant)
induced map $\tau^{\sharp }:W^{\Ga_2} \to W^{\Ga_1}$ on invariants.  It is given explicitly by: for $w \in W^{\Ga_2}$, 

\begin{equation} \label{transinvf}
\tau^{\sharp }(w)=w.\tau(\Ga_1). \end{equation}
\end{thm}

Note that this makes sense: $\tau(\Ga_1)$ is an $R$-linear combination of right
cosets of $\Ga_2$, and by Remark \ref{mean}, a right coset of $\Ga_2$
operating on the right of a $\Ga_2$-invariant element is meaningful.  It is
manifestly a $\Ga_1$-invariant element.

\begin{proof}
Since the functor $\Hom(\cdot,\cdot)$ is contravariant in its first argument, the given map $\tau:R[\Ga_1 \bs G] \to R[\Ga_2 \bs G]$ induces a $G$-map 
$\tau^{*}:\Hom_R(R[\Ga_2 \bs G],W) \to \Hom_R(R[\Ga_1 \bs G],W)$ given by
$\tau^{*}(f)=f \circ \tau$, hence by restriction a map
$\tau^{*}:\Hom_R(R[\Ga_2 \bs G],W)^G \to \Hom_R(R[\Ga_1 \bs G],W)^G$ on
$G$-invariants.  Then the composite 
$\varphi_{\Ga_1}^{-1}\circ\tau^{*}\circ\varphi_{\Ga_2}^{}$ is the desired induced map 
$\tau^{\sharp }$:

$$\xymatrix{
{W^{\Ga_2}}\ar^{\tau^{\sharp}}[r]\ar_{\varphi_{\Ga_2}}^{\sim}[d]&{W^{\Ga_1}}\ar^{\varphi_{\Ga_1}}_{\sim}[d]\\
{\Hom_R(R[\Ga_2 \bs G],W)^G}\ar_{\tau^*}[r]&{\Hom_R(R[\Ga_1 \bs G],W)^G}
}$$
To see that $\tau^{\sharp }$ is given by the formula \eqref{transinvf}, note that by
the formulas in Proposition \ref{compinv}, for $w \in W^{\Ga_2}$, we have that
$\varphi_{\Ga_1}^{-1}\circ\tau^{*}\circ\varphi_{\Ga_2}^{}(w)=
\varphi_{\Ga_1}^{-1}(\tau^{*}(\varphi_{\Ga_2}^{}(w)))=
\varphi_{\Ga_1}^{-1}(\varphi_{\Ga_2}^{}(w)\circ\tau)=
(\varphi_{\Ga_2}(w)\circ\tau)(\Ga_1)=(\varphi_{\Ga_2}(w))(\tau(\Ga_1))=
w.\tau(\Ga_1)$.
\end{proof}

\begin{rem} 

It is clear from its construction that the mapping is contravariantly functorial; i.e., if $\tau:R[\Ga_1 \bs G]\to R[\Ga_2 \bs G]$ and $\sigma:R[\Ga_2 \bs G]\to R[\Ga_3 \bs G]$ are right $G$-maps and $W$ is a right $G$-module, then 
$(\sigma\circ\tau)^{\sharp }=\tau^{\sharp }\circ\sigma^{\sharp }:W^{\Ga_3}\to W^{\Ga_1}$.
\end{rem}

\begin{prop}\label{compatinv}
Let $\Ga_1$, $\Ga_2$ be subgroups of $G$, let $W$, $X$ be right $G$-modules, and
let $\psi:W \to X$ be a right $G$-module map.  Suppose that $\tau:R[\Ga_1 \bs G] \to R[\Ga_2 \bs G]$ is a right $G$-module map.  Then the transplantations $\tau_W^{\sharp }:W^{\Ga_2}\to W^{\Ga_1}$ and $\tau_X^{\sharp }:X^{\Ga_2}\to X^{\Ga_1}$ commute with $\psi$, i.e., the diagram commutes:

$$\xymatrix{
{W^{\Ga_2}}\ar_{\psi}[d]\ar^{\tau_W^{\sharp}}[r]&{W^{\Ga_1}}\ar^{\psi}[d]\\
{X^{\Ga_2}}\ar_{\tau_X^{\sharp}}[r]&{X^{\Ga_1}}
}$$
\end{prop}

\begin{proof}
This is immediate from Formula \eqref{transinvf}.
\end{proof}

For an exposition of the use of Sunada's theorem in the
construction of isospectral manifolds, see \cite{Brooks}, \cite{Buser},
\cite{Gordon}, or \cite{Sunada}.

\subsection{Transplantation of coinvariants}\label{transcoinvsec}
Next, we turn to the formula for transplantation of coinvariants dual to that for
transplantation of invariants.  Again the formula is based upon two well known
isomorphisms.

Let $G$ be a group, let $\Ga\subseteq G$ be a subgroup, and $V$ be a left $R\Ga$-module.  One can associate naturally to $V$ a left $RG$-module by 
$RG \otimes_{R\Ga} V$; this is often called the \emph{induced module}.  The
$G$-module action is given by $g.(x\otimes v)=(gx)\otimes v$.

\begin{prop} \label{shapcoinv}
There is a natural isomorphism $$V_\Ga \to (RG \otimes_{R\Ga} V)_G$$ of the
$\Ga$-coinvariants of $V$ with the $G$-coinvariants of the induced module.  The
isomorphism is given explicitly as follows: to the residue class
$\overline{v} \in V_{\Ga}$, associate the element $\overline{1 \otimes v} \in (RG \otimes_{R\Ga} V)_G$.  Its inverse sends an element $\overline{x \otimes v} \in (RG \otimes_{R\Ga} V)_G$ to $\overline{v} \in V_{\Ga}$.
\end{prop}

\begin{proof}
Again, this is a special case of Shapiro's Lemma \cite{Brown}. 
\end{proof}

Now let $V$ be a left $RG$-module.  Form the $R$-module 
$R[\Ga \bs G]\otimes_R V$.  Here $G$ acts on the right on
$R[\Ga \bs G]$, hence on the left on $R[\Ga \bs G]$ by 
$g.(\Ga x)=\Ga xg^{-1}$.  One can then endow $R[\Ga \bs G]\otimes_R V$ with the diagonal left $G$-action by 
$g.(\Ga x \otimes v)=
g.(\Ga x) \otimes g.v=\Ga xg^{-1}\otimes gv$.  

One can also restrict the action of $G$ to a left action of $\Ga$ on $V$, then
form the induced module $RG \otimes_{R\Ga}\Res^G_{\Ga} V$.  As in the case of
invariants (Proposition \ref{twoactinv}), these two modules are isomorphic:

\begin{prop} \label{twoactcoinv}
There is a natural isomorphism of left $RG$-modules
$$RG \otimes_{R\Ga}\Res^G_{\Ga} V \to R[\Ga \bs G]\otimes_R V.$$
It is given as follows.  To 
$x \otimes v \in RG \otimes_{R\Ga}\Res^G_{\Ga} V$, ($x \in G$), associate the
element
$x(\Ga \otimes v)=\Ga x^{-1}\otimes xv \in R[\Ga \bs G]\otimes_R V$.  To
the element $\Ga x \otimes v \in R[\Ga \bs G]\otimes_R V$ ($x \in G$), the
inverse map associates the element $x^{-1}\otimes xv \in RG
\otimes_{R\Ga}\Res^G_{\Ga} V$.
\end{prop}

\begin{proof}
This is a straightforward verification.
\end{proof}

Passing to $G$-coinvariants yields
\begin{cor} \label{twoactcoinvG}
There is a natural isomorphism
$$(RG \otimes_{R\Ga}\Res^G_{\Ga} V)_G \to (R[\Ga \bs G]\otimes_R V)_G.$$
\end{cor}

As in the case of invariants, one can compose the isomorphisms of Proposition
\ref{shapcoinv} and Corollary \ref{twoactcoinvG} to obtain

\begin{prop} \label{compcoinv}
Let $\Ga$ be a subgroup of $G$, $V$ a left $RG$-module.  Then there is a natural
isomorphism 
$$\varphi_{\Ga}:V_{\Ga} \to (R[\Ga \bs G]\otimes_R V)_G$$ given as follows.
For $v \in V$, $\varphi_{\Ga}(\overline{v})=\overline{\Ga \otimes v}$.
For $x \in G$, $\varphi_{\Ga}^{-1}(\overline{\Ga x \otimes v})=\overline{xv}$.
\end{prop}

\begin{proof}
This is clear.  For the formula for $\varphi_{\Ga}^{-1}$, note that in the module of coinvariants $(R[\Ga \bs G]\otimes_R V)_G$, for $x \in G$ we have by definition of the diagonal action $\overline{\Ga x \otimes v}=\overline{x^{-1}.(\Ga \otimes xv)}=
\overline{\Ga \otimes xv}$.
\end{proof}

\begin{thm} \label{transcoinv} 
(Transplantation of coinvariants)
Let $\Ga_1$ and $\Ga_2$ be subgroups of a group $G$, and let $V$ be a left $RG$-module.  Suppose that there is a map of right $RG$-modules $\tau:R[\Ga_1 \bs G] \to
R[\Ga_2 \bs G]$.  Then there is a covariantly functorial induced $R$-linear map $\tau_{\sharp }:V_{\Ga_1} \to V_{\Ga_2}$.  It is given as follows.  For $v \in V$,
\begin{equation} \label{transcoinvf}
\tau_{\sharp }(\overline v)=\overline{\tau(\Ga_1).v}.
\end{equation}

\end{thm}

Note that this expression makes sense: $\tau(\Ga_1)$ is an $R$-linear combination of right cosets of $\Ga_2$, so by Remark \ref{mean}, such an expression is meaningful.

\begin{proof}
The map $\tau:R[\Ga_1 \bs G] \to R[\Ga_2 \bs G]$ induces a map
$$\tau \ot \Id:R[\Ga_1\bs G]\ot_RV \to R[\Ga_2\bs G]\ot_RV.$$  The desired
transplantation map $\tau_{\sharp }$ is then defined by the following commutative
diagram:

$$\xymatrix{
{V_{\Ga_1}}\ar^{\tau_{\sharp}}[r]\ar_{\varphi_{\Ga_1}}^{\sim}[d]&{V_{\Ga_2}}\ar^{\varphi_{\Ga_2}}_{\sim}[d]\\
{(R[\Ga_1\bs G]\ot_RV)_G}\ar_{\tau\otimes\Id}[r]&{(R[\Ga_2\bs G]\ot_RV)_G.}
}$$

In view of the explicit description of the vertical isomorphisms $\varphi_{\Ga_i}$
furnished by Proposition \ref{compcoinv} above, the verification of the formula
\eqref{transcoinvf} for $\tau_{\sharp }$ is immediate.
\end{proof}

As in the case of invariants, it is immediate from the explicit formula that
transplantation of coinvariants commutes with $G$-maps:

\begin{prop} \label{compatcoinv}
Let $\Ga_1$, $\Ga_2$ be subgroups of $G$, let $V$, $X$ be left $G$-modules, and
let
$\psi:V \to X$ a left $G$-module map.  Suppose that 
$\tau:R[\Ga_1 \bs G] \to R[\Ga_2 \bs G]$ is a right $G$-module map.  Then the
transplantations $\tau^V_{\sharp }:V_{\Ga_1}\to V_{\Ga_2}$ and
$\tau^X_{\sharp }:X_{\Ga_1}\to X_{\Ga_2}$ commute with $\psi$, i.e., the diagram
commutes:

$$\xymatrix{
{V_{\Ga_1}}\ar^{\tau^V_{\sharp}}[r]\ar_{\psi}[d]&{V_{\Ga_2}}\ar^{\psi}[d]\\
{X_{\Ga_1}}\ar_{\tau^X_{\sharp}}[r]&{X_{\Ga_2}}
}$$

\end{prop}

\subsection{Transplantation of singular chains}\label{transsingchains}
As an application of Theorem \ref{transcoinv}, we consider transplantation of
chains, cycles, and homology.  The following lemma is easy to verify.

\begin{lem} \label{algfact} 
Let $\Ga$ be a group, and let $X$ be a left $\Ga$-set.  Then the free $\Z$-module 
$\Z[\Ga \bs X]$ on the orbit set is isomorphic to the $\Z$-module of coinvariants $\Z[X]_{\Ga}$.
\end{lem}

For a compact Riemannian manifold $M$, let $\Delta_q(M)$ denote the set of
singular $q$-simplices on $M$.  Then the $\Z$-module of singular $q$-chains on $M$ is given by $C_q(M)=\Z[\Delta_q(M)]$.  Now let $G$ be a finite group of isometries of $M$, and let $\Ga_1, \Ga_2 \subseteq G$ be subgroups.  Suppose that there is a right $\Z G$-module map $\tau:\Z[\Ga_1\bs G] \to \Z[\Ga_2\bs G]$.  We will show that there is an induced transplantation of chains $\tau_{\sharp }:C_q(\Ga_1 \bs M) \to C_q(\Ga_2 \bs M)$ on the orbit spaces.  To see this, note that for $\Ga=\Ga_1$ or $\Ga_2$, the set of orbits $\Ga \bs \Delta_q(M)$ is $\Delta_q(\Ga \bs M)$, by covering space theory.  By the lemma, $C_q(\Ga \bs M)=\Z[\Delta_q(\Ga \bs M)] \cong
\Z[\Delta_q(M)]_{\Ga}=C_q(M)_{\Ga}$.  Thus Theorem \ref{transcoinv} immediately
yields a transplantation map $\tau_{\sharp }:C_q(\Ga_1 \bs M) \to C_q(\Ga_2 \bs M)$ on $q$-chains. Moreover, since the singular boundary maps are natural and hence commute with the $G$-action, Proposition \ref{compatcoinv} shows that the transplantation of cycles commutes with boundaries and hence carries cycles to cycles, boundaries to boundaries, etc.  Thus there is an induced transplantation $\tau_{\sharp}:H_q(\Ga_1 \bs M) \to H_q(\Ga_2 \bs M)$ on homology.

\subsection{Pairing lemma: duality of transplantations of invariants and coinvariants}
To compare the intermediate Jacobians of two Sunada isospectral compact K\"ahler manifolds, we will use transplantation of invariants.  We will need the fact that the transplantation of invariants and transplantation of coinvariants enjoy a
duality relation relative to certain pairings.

\begin{dfn}
Let $G$ be a group, let $W$ be a right $\C G$-module, and let $V$ be a left $\Z G$-module.  A $\Z$-bilinear pairing $\langle\cdot,\cdot\rangle:W \times V \to \C$ is 
{\emph{$G$-balanced}} if for all $w \in W$, $v \in V$, and $g \in G$, we have
$\langle w,gv\rangle=\langle wg,v\rangle$.
\end{dfn}

\begin{exa} \label{integration}
Let $M$ be a compact oriented manifold, and let $G$ be a subgroup of the group of diffeomorphisms of $M$.  Let $V=C_q(M)$, the $\Z$-module of smooth singular $q$-chains, let $W=\ms{A}^q_{\C}(M)$ be the space of $q$-forms, and let $\langle\cdot,\cdot\rangle:\ms{A}^q_{\C}(M)\times C_q(M)\to\C$ be the integration pairing $\langle\omega,c\rangle=\int_c\,\omega$.  Then $G$ acts on $C_q(M)$ on the left by composition and on $\ms{A}^q_{\C}(M)$ on the right by pullback.  The pairing is clearly $G$-balanced, as$\int_{gc}\,\omega=\int_c\,g^*\omega$.
\end{exa}

Let $\Ga$ be a subgroup of $G$.  A $G$-balanced pairing $\langle\cdot,\cdot\rangle:W
\times V \to \C$ induces a pairing $\langle\langle\cdot,\cdot\rangle\rangle:W^{\Ga}
\times V_{\Ga} \to \C$ given by $\langle\langle w,\overline{v}\rangle\rangle=\langle w,v\rangle$.  Indeed, the value of $\langle w,v\rangle$ depends only upon the residue class $\overline{v}\in V_{\Ga}$, since for $\gamma\in\Ga$, $\langle w,\gamma v\rangle=\langle w\gamma,v\rangle=\langle w,v\rangle$, as $w$ is $\Ga$-invariant.  The pairing $\langle\langle\cdot,\cdot\rangle\rangle$ induces a map $\flat_{\Ga}:W^{\Ga}\to (V_{\Ga})^*$ sending $w\in W^{\Ga}$ to the complex linear functional given by $\ol{v}\mapsto\langle\langle w,\ol{v}\rangle\rangle=\langle w,v\rangle$.

Let $G$ be a group, let $W$ be a right $\C G$-module, and let $V$ a left $\Z G$-module.  Let $\langle\cdot,\cdot\rangle:W \times V \to \C$ be a $G$-balanced pairing as above, and let $\Ga_1$, $\Ga_2$ be subgroups of $G$.  Then there are maps
$\flat_{\Ga_1}:W^{\Ga_1}\to (V_{\Ga_1})^*$ and $\flat_{\Ga_2}:W^{\Ga_2}\to (V_{\Ga_2})^*$ as described above.  Let $\tau:\Z[\Ga_1 \bs G] \to \Z[\Ga_1 \bs G]$ be a map of right $\Z G$-modules.  Then by transplantation of coinvariants (Theorem \ref{transcoinv}), $\tau$ induces $\tau_{\sharp }:V_{\Ga_1} \to V_{\Ga_2}$, which in turn induces a map backwards on the duals, $(\tau_{\sharp})^*:(V_{\Ga_2})^*\to(V_{\Ga_2})^*$ (here $(V_{\Ga_i})^*$ denotes the complex dual $\Hom(V_{\Ga_i},\C)$).  By Theorem \ref{transinv}, $\tau$ also induces the transplantation map $\tau^{\sharp }:W^{\Ga_2} \to W^{\Ga_1}$ on invariants.

\begin{lem} \label{pairing} (Pairing lemma)
Transplantation is compatible with the pairings; i.e., the diagram

$$\xymatrix{
{W^{\Ga_2}}\ar_{\tau^{\sharp}}[d]\ar^{\flat_{\Ga_2}}[r]&{(V_{\Ga_2})^*}\ar^{(\tau_{\sharp})^*}[d]\\
{W^{\Ga_1}}\ar_{\flat_{\Ga_1}}[r]&{(V_{\Ga_1})^*}
}$$

commutes.
\end{lem}

\begin{proof}
This is an immediate verification, using the $G$-balanced property and the
explicit formulas \eqref{transinvf} and \eqref{transcoinvf} for the two
transplantations.
\end{proof}

\section{Proofs of the Main Results}
We will use transplantation to prove Theorems~\ref{thm.mt} and \ref{thm.pq}.  We first record a lemma from \cite{GordonMakoverWebb}.

\begin{lem}\label{EquivC-repsDefQ}
Let $G$ be a finite group, and let $V_1$ and $V_2$ be finite dimensional rational
representations of $G$ (i.e., $\Q G$-modules).  Suppose that the complexified
representations $\C\otimes_{\Q} V_1$ and $\C\otimes_{\Q} V_2$ are equivalent as
$\C G$-modules.  Then $V_1$ and $V_2$ were already equivalent as $\Q G$-modules.
\end{lem}

Recalling Remark~\ref{rem.jac}, we can now establish Theorem~\ref{thm.mt}, which we first restate in the language of representations.

\begin{thm}\label{thm.mt'}
Let $X$ be a compact complex manifold of K\"ahler type and let $G$ be a finite group of biholomorphic transformations of $X$.  Let $\Ga_1$ and $\Ga_2$ be subgroups of $G$ that act freely on $X$, and let $X_i=\Ga_i\bs X$ for $i=1,2$.  Suppose that the complex linear right representations $\C[\Ga_1 \bs G]$ and $\C[\Ga_2 \bs G]$ of $G$ afforded by the two right coset spaces are equivalent (see Section \ref{notation} for notation and definitions).  Then for $1\le k\le d$, the intermediate Jacobians $J^k(X_1)$ and $J^k(X_2)$ are isogenous.
\end{thm}

\begin{rem} Here $J^k(X_i)$ is the $k$th Griffiths intermediate Jacobian, as defined above in Definition~\ref{def.jac}.  The reader familiar with the Weil intermediate Jacobians $J_W^k(X)$ (see \cite{Weil} or \cite{Lewis}) can easily adapt the proof below to prove isogeny of $J_W^k(X_1)$ and $J_W^k(X_2)$.

\end{rem}

\begin{proof}
For $i=1,2$, the complex permutation representation $\C[\Ga_i\bs G]$ is just the complexification $\C\otimes_{\Q}\Q[\Ga_i\bs G]$ of the rational permutation representation $\Q[\Ga_i\bs G]$, so by the lemma, there is an equivalence $\Q[\Ga_1\bs G]\to\Q[\Ga_2\bs G]$ of rational permutation representations.  Multiplying by a large enough integer to clear denominators yields a $\Z G$-module map $\tau:\Z[\Ga_1\bs G]\to\Z[\Ga_2\bs G]$.  Clearing denominators in the inverse equivalence $\Q[\Ga_2\bs G]\to\Q[\Ga_1\bs G]$ yields a $\Z G$-module map $\sigma:\Z[\Ga_2\bs G]\to\Z[\Ga_1\bs G]$; while $\sigma$ and $\tau$ are not inverses, their composites $\sigma\circ\tau$ and $\tau\circ\sigma$ are just multiplication by an integer.

The group $G$ acts on $X$ on the left, hence on the right by pullback on the complex-valued differential $(2k-1)$-forms $\ms{A}_{\C}^{2k-1}(X)$.  For $i=1,2$, the $\Ga_i$-invariant forms are just the forms on $X_i=\Ga_i\bs X$, i.e., $\ms{A}_{\C}^{2k-1}(X)^{\Ga_i}=\ms{A}_{\C}^{2k-1}(X_i)$.  The transplantation of invariants $\tau^{\sharp}:\ms{A}_{\C}^{2k-1}(X_2)\to\ms{A}_{\C}^{2k-1}(X_1)$ given by Theorem \ref{transinv} is compatible with the exterior derivative by Theorem \ref{compatinv} and hence carries closed (respectively, exact) forms into closed (respectively, exact) forms, so it induces a (contravariant) transplantation map $\tau^{\sharp}:H^{2k-1}(X_2,\C)\to H^{2k-1}(X_1,\C)$ on cohomology.  For the same reason, Proposition \ref{FiltrationCompat} shows that $\tau^{\sharp}$ respects the Hodge filtration, so it restricts to a transplantation mapping $\tau^{\sharp}:F^k H^{2k-1}(X_2,\C)\to F^k H^{2k-1}(X_1,\C)$ making the left square in the diagram 
\begin{equation}\label{H/FH}
\xymatrix@R=30pt@C=20pt{
{0}\ar[r]&{F^k H^{2k-1}(X_2,\C)}\ar[r]\ar_{\tau^{\sharp}}[d]&{H^{2k-1}(X_2,\C)}\ar[r]\ar_{\tau^{\sharp}}[d]&{H^{2k-1}(X_2,\C)/F^k H^{2k-1}(X_2,\C)}\ar[r]\ar_{\tau^{\sharp}}@{-->}[d]&{0}\\
{0}\ar[r]&{F^k H^{2k-1}(X_1,\C)}\ar[r]&{H^{2k-1}(X_1,\C)}\ar[r]&{H^{2k-1}(X_1,\C)/F^k H^{2k-1}(X_1,\C)}\ar[r]&{0}}
\end{equation}
commute and inducing the right vertical arrow on the quotients.  Temporarily denote by $V_k(X_i)$ the Hodge filtration quotient appearing in the definition of $J^k(X_i)$: i.e., $V_k(X_i)=H^{2k-1}(X_i,\C)/F^k H^{2k-1}(X_i,\C)$, $i=1,2$, so that $J^k(X_i)=V_k(X_i)/H^{2k-1}(X_i,\Z)$.   

The boundary maps in the chain complex of smooth singular chains are maps of $G$-modules, so the transplantation of coinvariants 
$$\xymatrix@R=8pt{
{C_{2k-1}(X_1)}\ar@{=}[d]&{C_{2k-1}(X_2)}\ar@{=}[d]\\
{C_{2k-1}(\Ga_1\bs X)}\ar@{=}[d]&{C_{2k-1}(\Ga_2\bs X)}\ar@{=}[d]\\
{C_{2k-1}(X)_{\Ga_1}}\ar^{\tau_{\sharp}}[r]&{C_{2k-1}(X)_{\Ga_2}}
}$$
arising from Theorem \ref{transcoinv} and section \ref{transsingchains} carries cycles to cycles and boundaries to boundaries, and hence descends to a transplantation map $\tau_{\sharp}:H_{2k-1}(X_1)\to H_{2k-1}(X_2)$ on integer homology.  

Consider the diagram
$$\xymatrix@R=8pt{
{H^{2k-1}(X_2,\Z)}\ar@{=}[d]\\
{\Hom(H_{2k-1}(X_2),\Z)}\ar_{(\tau_{\sharp})^*}[ddd]\ar@{^{(}->}[r]&{\Hom(H_{2k-1}(X_2),\C)}\ar_{(\tau_{\sharp})^*}[ddd]\ar^{\quad\flat_{\Ga_2}^{-1}}[r]&{H^{2k-1}(X_2,\C)}\ar_{\tau^{\sharp}}[ddd]\ar@{->>}[r]&{V_k(X_2)}\ar_{\tau^{\sharp}}[ddd]\\\\\\
{\Hom(H_{2k-1}(X_1),\Z)}\ar@{^{(}->}[r]&{\Hom(H_{2k-1}(X_1),\C)}\ar^{\quad\flat_{\Ga_1}^{-1}}[r]&{H^{2k-1}(X_1,\C)}\ar@{->>}[r]&{V_k(X_1)}\\
{H^{2k-1}(X_1,\Z)}\ar@{=}[u]
}$$
which we claim commutes.  The left square obviously commutes.  The center square commutes, by Lemma \ref{pairing} (the Kronecker pairing of $H_{2k-1}(X_i)$ with $H^{2k-1}(X_i,\C)$ is induced by the integration pairing of a $(2k-1)$-chain with a $(2k-1)$-form of Example \ref{integration}; moreover, the Kronecker pairing is nondegenerate, so the maps $\flat_{\Ga_i}$ are isomorphisms), while the right square commutes because it is the right square of the diagram \eqref{H/FH} above.  This establishes the claim.  For $i=1,2$, let $f_i:H^{2k-1}(X_i,\Z)\to V_k(X_i)$ be the composite along each row of the diagram.

Now consider the diagram
$$\xymatrix{
{0}\ar[r]&{H^{2k-1}(X_2,\Z)}\ar_{(\tau_{\sharp})^*}[d]\ar^{\quad f_2}[r]&{V_k(X_2)}\ar_{\tau^{\sharp}}[d]\ar@{->>}[r]&{J^k(X_2)}\ar_{\hat{\tau}}@{-->}[d]\ar[r]&{0}\\
{0}\ar[r]&{H^{2k-1}(X_1,\Z)}\ar_{\quad f_1}[r]&{V_k(X_1)}\ar@{->>}[r]&{J^k(X_1)}\ar[r]&{0}
}$$
whose rows define the $k$th intermediate Jacobians $J^k(X_i)$.  The commutativity of the left square induces the right vertical map $\hat{\tau}:J^k(X_2)\to J^k(X_1)$ making the diagram commute.  By the functoriality of the transplantation construction, the left vertical map $(\tau_{\sharp})^*$ is a rational isomorphism of free finite-rank $\Z$-modules, hence is injective with finite cokernel.  The middle vertical map is an isomorphism of complex vector spaces.  By the Snake Lemma, the cokernel of the left vertical map $(\tau_{\sharp})^*$ is isomorphic to the kernel of the right vertical map $\hat{\tau}$, so the latter is the desired isogeny of $J^k(X_2)$ with $J^k(X_1)$. 

\end{proof}

We next prove Theorem~\ref{thm.pq}, which we also restate in the language of representations.

\begin{thm}\label{thm.pq'} Let $X$ be a compact $n$-dimensional K\"ahler manifold and let $G$ be a finite group of holomorphic isometries of
$X$.  Let $\Ga_1$ and $\Ga_2$ be subgroups of $G$ that act freely on $X$ and let $X_i=\Ga_i\bs X$, $i=1,2$.  Suppose that the complex linear right representations $\C[\Ga_1 \bs G]$ and
$\C[\Ga_2 \bs G]$ of $G$ afforded by the right coset spaces are equivalent.  Then for each pair of non-negative integers $(p,q)$ with $p+q\leq n$, the Hodge Laplacians on $\ms{A}^{p,q}(\Ga_1\bs X)$ and $\ms{A}^{p,q}(\Ga_2\bs X)$ are isospectral.

\end{thm}

\begin{proof} The group $G$ acts on the right by pullback on the complex-valued differential forms $\ms{A}^*(X)$.  The hypothesis that $G$ acts holomorphically on $X$ implies that $G$ preserves each of the subspaces $\ms{A}^{p,q}(X)$.  The hypothesis that $G$ acts by isometries on $X$ implies that the action of $G$ commutes with the Hodge Laplacian $\Delta: \ms{A}^{p,q}(X)\to \ms{A}^{p,q}(X)$, i.e., $\Delta$ is a $G$-module map.   For $i=1,2$, the $\Ga_i$-invariant forms are the forms on $X_i=\Ga_i\bs X$, i.e., $\ms{A}^{p,q}(X)^{\Ga_i}=\ms{A}^{p,q}(X_i)$.  By Theorem \ref{compatinv}, the transplantation of invariants $\tau^{\sharp}:\ms{A}^{p,q}(X_2)\to\ms{A}^{p,q}(X_1)$ given by Theorem \ref{transinv} intertwines the Hodge Laplacians on $\ms{A}^{p,q}(X_2)$ and $\ms{A}^{p,q}(X_1)$, thus proving Theorem~\ref{thm.pq'}.

\end{proof}

\section{Examples}\label{sec.examples}

We give examples of K\"ahler manifolds and group actions satisfying the hypotheses of Thereom~\ref{thm.pq'} and therefore also of Theorem~\ref{thm.mt'}.

\begin{exa}\label{Liz&Naveed}\cite{Brooks}
By Serre's Theorem \cite{Serre}, any finite group $G$ arises as the fundamental group of a compact K\"ahler manifold.  In particular, let $p$ be an odd prime, and let $X_0$ be a compact K\"ahler manifold whose fundamental group is the symmetric group $S_{p^3}$ of permutations of $p^3$ letters; let $X$ be the universal cover of $X_0$ with the pullback K\"ahler structure.  The fundamental group $S_{p^3}$ of $X_0$ acts by K\"ahler isometries on $X$ as deck transformations.  Let $E=\Z_p\times\Z_p\times\Z_p$ be the $p$-elementary abelian group of order $p^3$, and let 
$$H=\left\{\bmat{1&x&z\\0&1&y\\0&0&1}:x,y,z\in\Z_p\right\}$$
be the Heisenberg group over the $p$-element field $\Z_p$.  Each of $E$ and $H$ acts on itself by left translation, so admits a Cayley embedding in $S_{p^3}$; let $\Ga_1$ (respectively, $\Ga_2$) be the image of $E$ (respectively, $H$) under its Cayley embedding into $S_{p^3}$.  It is clear that $\Ga_1$ and $\Ga_2$ are almost conjugate subgroups of $S_{p^3}$: indeed, any bijection of $\Ga_1$ with $\Ga_2$ sending the identity to the identity will carry any element $\gamma$ of $\Ga_1$ to an element of $\Ga_2$ conjugate to $\gamma$ in $S_{p^3}$, since the two permutations have the same cycle structure.  Thus $\Ga_1$ and $\Ga_2$ are representation equivalent subgroups of $S_{p^3}$, so the orbit spaces $X_1=\Ga_1\bs X$ and $X_1=\Ga_1\bs X$ are isospectral K\"ahler manifolds.  They are not isometric, since their fundamental groups $\Ga_1$ and $\Ga_2$ are not isomorphic: $\Ga_1$ is abelian, while $\Ga_2$ is not.
\end{exa}

\begin{rem}
By adapting this construction as in \cite{Brooks} (see also \cite{ShamsStanhopeWebb}), one can construct arbitrarily large families of isospectral, mutually nonisometric compact K\"ahler manifolds.
\end{rem}

\begin{exa}\label{McReynolds}
In \cite{McReynolds}, D. B. McReynolds constructs arbitrarily large families of isospectral, mutually nonisometric Hermitian locally symmetric spaces, all of which, being quotients of Hermitian symmetric spaces (which are K\"ahler by Chapter VIII, Proposition 4.1 of \cite{Helgason}) by discrete subgroups acting by K\"ahler automorphisms, are therefore compact K\"ahler manifolds.  
\end{exa}

\begin{exa}\label{MiatelloPodesta1}
In \cite{MiatelloPodesta1}, R. Miatello and R. Podest\'a exhibit examples of pairs of isospectral but nonisometric Bieberbach K\"ahler manifolds.
\end{exa}

\section{The Lazzeri Jacobian}
Another natural generalization of the isogeny of Sunada isospectral Riemann surfaces arises in the context of the Lazzeri Jacobian.

\begin{dfn}
Let $M$ be a compact oriented Riemannian manifold of dimension $n=2m$ where $m$ is an odd integer.  The \emph{Lazzeri Jacobian} $J_L(M)$ is a complex torus defined as follows: $J_L(M)=H^m(M,\R)/H^m(M,\Z)$, where the complex structure is given by the Hodge $*$-operator determined by the metric and the orientation, which, because of the restriction on the dimension, satisfies $**=-\Id$ on the middle-dimensional forms $\ms{A}^m(M)$; the Hodge $*$-operator determines a complex structure on the middle-dimensional cohomology via the Hodge isomorphism of $H^m(M,\Z)$ with the space of harmonic $m$-forms.
\end{dfn}

\begin{rem} The Lazzeri Jacobian is a principally polarized abelian variety, with a polarization whose imaginary part is defined as $(\alpha,\beta)=-\int_M \alpha\wedge\beta$.  See \cite{Rubei} or \cite{BirkenhakeLangeCT} for more information about Lazzeri Jacobians.
\end{rem}

\begin{thm}
Let $M$ be a compact oriented Riemannian manifold of dimension $2m$ where $m$ is odd, and let $G$ be a finite group of orientation-preserving isometries of $M$.  Let $\Ga_1$ and $\Ga_2$ be subgroups of $G$ that act freely on $M$, and let $M_i=\Ga_i\bs M$ for $i=1,2$.  Suppose that the complex linear right representations $\C[\Ga_1 \bs G]$ and $\C[\Ga_2 \bs G]$ of $G$ afforded by the two right coset spaces are equivalent.  Then the Lazzeri Jacobians $J_L(M_1)$ and $J_L(M_2)$ are isogenous complex tori.
\end{thm}

\begin{proof}
The proof is virtually the same as that of Theorem \ref{thm.mt'}.  The equivalence of the two complex representations that were defined over the rationals implies equivalence as $\Q G$-modules, by Lemma \ref{EquivC-repsDefQ}.  Clearing denominators as before yields maps of $\Z G$-modules $\tau:\Z[\Ga_1\bs G]\to\Z[\Ga_2\bs G]$ and $\sigma:\Z[\Ga_2\bs G]\to\Z[\Ga_1\bs G]$ whose composites are multiplication by an integer.  Then $\tau$ induces (via transplantation of invariants) a cohomology transplantation $\tau^{\sharp}:H^m(M_2,\R)\to H^m(M_1,\R)$ and (via transplantation of coinvariants) a transplantation $\tau_{\sharp}:H_m(M_1)\to H_m(M_2)$ of singular chains, as before.  The diagram
$$\xymatrix@C=8pt{
{H^m(M_2,\Z)}\ar@{=}[r]&{\Hom(H_m(M_2),\Z)}\ar_{(\tau_{\sharp})^*}[d]\ar[rr]&&{\Hom(H_m(M_2),\R)}\ar_{(\tau_{\sharp})^*}[d]\ar^{\qquad\flat_2^{-1}}[rr]&&{H^m(M_2,\R)}\ar^{\tau^{\sharp}}[d]\\
{H^m(M_1,\Z)}\ar@{=}[r]&{\Hom(H_m(M_1),\Z)}\ar[rr]&&{\Hom(H_m(M_2),\R)}\ar_{\qquad\flat_1^{-1}}[rr]&&{H^m(M_1,\R)}
}$$
commutes, by the pairing lemma.  The transplantation $\tau^{\sharp}:H^m(M_2,\R)\to H^m(M_1,\R)$ is complex linear by Lemma \ref{compatinv}, since $G$ acts by orientation-preserving isometries and hence preserves the Hodge $*$-operator, which defines the complex structure on the cohomology.  The remainder of the proof is as before, using the commutative diagram
$$\xymatrix{
{0}\ar[r]&{H^m(M_2,\Z)}\ar[r]\ar^{(\tau_{\sharp})^*}[d]&{H^m(M_2,\R)}\ar^{\tau^{\sharp}}[d]\ar[r]&{J_L(M_2)}\ar^{\hat{\tau}}@{-->}[d]\ar[r]&{0}\\
{0}\ar[r]&{H^m(M_1,\Z)}\ar[r]&{H^m(M_1,\R)}\ar[r]&{J_L(M_1)}\ar[r]&{0}
}$$
and the fact that the left vertical arrow is a rational isomorphism, hence an injection with finite cokernel, to induce the desired isogeny $\hat{\tau}$ of Lazzeri Jacobians.
\end{proof}


\end{document}